\documentclass[hidelinks,12pt]{article}

\usepackage[margin=1in]{geometry}  

\usepackage{BOONDOX-cal}
\usepackage{amsmath}               
\usepackage{amsfonts}              
\usepackage{amsthm}                
\usepackage{amssymb}
\usepackage{titlesec}
\usepackage{verbatim}
\usepackage{hyperref}
\usepackage{theoremref}

\usepackage{bm}

\newtheoremstyle{break}
  {\topsep}{\topsep}%
  {\itshape}{}%
  {\bfseries}{}%
  {\newline}{}%
\theoremstyle{break}
\newtheorem{thm}{Theorem}

\theoremstyle{definition}
\newtheorem{lemma}[thm]{Lemma}

\theoremstyle{definition}
\newtheorem{corollary}[thm]{Corollary}

\theoremstyle{definition}
\newtheorem{remark}{Remark}

\newtheorem{mydef}{Definition}

\newcommand{\RR}{\mathbb{R}}      
\newcommand{\ZZ}{\mathbb{Z}}      
\newcommand{\CC}{\mathbb{C}}
\newcommand{\D}{\text{d}}
\newcommand{\IFF}{if and only if }
\usepackage{physics}

\numberwithin{equation}{section}

\begin{document}

\title{\textbf{On the commutation properties of finite convolution and differential operators I: commutation}}

\author{Yury Grabovsky, \qquad Narek Hovsepyan}

\date{}
\maketitle

\begin{abstract}
  The commutation relation $KL = LK$ between finite convolution integral
  operator $K$ and differential operator $L$ has implications for spectral
  properties of $K$. We characterize all operators $K$ admitting this
  commutation relation. Our analysis places no symmetry constraints on the
  kernel of $K$ extending the well-known results of Morrison for real
  self-adjoint finite convolution integral operators.
\end{abstract}

\tableofcontents

\section{Introduction}

The need to understand spectral properties of finite convolution integral operators

\begin{equation} \label{K}
(Ku)(x) = \int_{-1}^1 k(x-y) u(y) \D y 
\end{equation}

\noindent acting on $L^2(-1,1)$ arises in a number of applications, including
optics \cite{frieden}, radio astronomy \cite{bracewell riddle},
\cite{bracewell wernecke}, electron microscopy \cite{gechberg}, x-ray
tomography \cite{grunbaum2}, \cite{tam}, noise theory \cite{davenport root}
and medical imaging \cite{al-aifari katsevich}, \cite{katsevich1},
\cite{katsevich2}, \cite{katsevich tovbis}. In some cases it is possible to
find a differential operator $L$ which commutes with $K$ (cf. \cite{pollak
  slepian,morrison,widom,katsevich1}),

\begin{equation} \label{C3}
KL = L K. \tag{C1}
\end{equation}

\noindent In this case eigenfunctions of $K$ can be chosen to be solutions of
ordinary differential equations. More precisely, \eqref{C3} implies that
eigenspaces $E_\lambda$ of $K$ are invariant under $L$, i.e. $L: E_\lambda
\mapsto E_\lambda$. Now if $L$ is diagonalizable, e.g. self-adjoint, or more
generally, normal (for characterization of normality see Remark~\ref{normal}),
then one can choose a basis for $E_\lambda$ consisting of eigenfunctions of
$L$. This permits to bring the vast literature on asymptotic properties of
solutions of ordinary differential equations to bear on obtaining analytical
information about the asymptotics of eigenvalues and eigenfunctions of
integral operators. With this said, we will not be investigating spectral
properties of differential operators that commute with integral operators. In
our view questions about differential operators are much more tractable than
questions about the integral operators, see e.g. \cite{zettl}, and our goal is
to find all connections between the two questions.

The most famous example of this phenomenon is the band-and time limited
operator of Landau, Pollak, and Slepian \cite{landau pollak 1}, \cite{landau
  pollak 2}, \cite{pollak slepian}--\cite{slepian2}, corresponding to $k(z) =
\frac{\sin (a z)}{z}$ in \eqref{K} with $a > 0$. Sharp estimates for
asymptotics of the eigenvalues of $K$ were derived using its commutation with
a second order symmetric differential operator, whose eigenfunctions are the
well-known prolate spheroidal wave functions that first appeared in the
context of quantum mechanics \cite{Morse53}. Another example is the
result of Widom \cite{widom}, where using comparison with special operators
that commute with differential operators, the author obtained asymptotic
behavior of the eigenvalues of a large class of integral operators with real-valued even kernels. A complete characterization of such special operators commuting with symmetric second order differential operators was achieved by Morrison \cite{morrison} (see also \cite{wright}, \cite{grunbaum}). We are interested in the possibility of extension of these ideas to the case of complex-valued $k(z)$. In this more general context the property of commutation must also be generalized, so as to permit the characterization of eigenfunctions as solutions of an eigenvalue problem for a second or fourth order differential operator. 

In this paper we analyze the commutation relation \eqref{C3}, under the
assumption that $k$ is analytic at the origin as in \cite{morrison},
\cite{wright}, or it has a simple\footnote{It is not hard to show that
  commutation is not possible for higher order poles.} pole at $0$, in which case the integral is
understood in the principal value sense (cf. Theorem~\ref{THM
  commutation}). Further, we consider extensions of the notion of commutation,
that also link integral equations with ordinary differential equations. A
natural extension of commutation, as explained in the introductory section in
\cite{aifariPhD} is 

\begin{equation} \label{C2}
\begin{cases}
K L_1 = L_2 K
\\
L_j^* = L_j, \qquad j=1,2
\end{cases},
\tag{C2}
\end{equation}

\noindent where $L_j, \ j=1,2$ are differential operators with complex coefficients. This has implications for singular value decomposition of $K$. It is easy to check that \eqref{C2} reduces to a commutation relation for $K^* K$, indeed we have

\begin{equation} \label{L1 K* K = K* K L1}
L_1 K^*K = K^* K L_1 ,
\end{equation}

\noindent and therefore singular functions of $K$ satisfy ODEs, in the sense
explained above. In fact, commuting pairs $(K, L)$ can also provide instances
of \eqref{C2}, as was observed in \cite{al-aifari katsevich},
\cite{katsevich1}, \cite{katsevich2}, \cite{katsevich tovbis} in applications
to truncated Hilbert transform operators ($k(z) = 1/z$). In this setting the
input function is considered on one interval while the output of $K$ is
defined on a different interval. Commutation relation of type \eqref{C2} is
obtained from \eqref{C3} by restricting the differential operator to corresponding intervals. Their method requires that $L$ has
real valued coefficients, while such constraint is not necessary to pass from
\eqref{C2} to \eqref{L1 K* K = K* K L1}. As a consequence a singular value
decomposition can be obtained for complex-valued operators $K$, as well.

 When $k(z)$ has a simple pole at the origin, the operator $K$ is not compact
 and may have a continuous spectrum (cf. \cite{koppelman
   pincus}). However, when we consider situations where the output of $K$ is
 defined on some other
 line segment in the complex plane, as in the examples of truncated Hilbert
 transform operators mentioned above, we uncover a rich set of operators $K$, such
 that $K^* K$ has discrete spectrum and singular value decomposition for $K$
 can be obtained following the ideas of  \cite{al-aifari katsevich},
 \cite{katsevich1}, \cite{katsevich2}, \cite{katsevich tovbis}. As an example of
 application of some of our results, the operator with kernel $k(z) = 1 / \sin
 \left( \frac{\pi}{8}z\right)$ considered from $L^2(-1,1) \to L^2(3,5)$ has a
 discrete singular value
decomposition
(see Remark~\ref{REM discrete spectrum} for details and more examples).

In the second part of this work \cite{gr hov sesqui} we consider a new kind of
commutation relation $\overline{K} L_1 = L_2 K$, with $L_j^T = L_j$ for
$j=1,2$, which we call \emph{sesquicommutation}. In this case the eigenspaces
of $K^* K$ will be shown to be invariant under the fourth order self-adjoint
operator $L_1^* L_1$. In the case of sesquicommutation we will be able to
prove that no nontrivial cases arise unless $L_{1}=L_{2}$. This reduction does
not work for \eqref{C2}, and hence only the case \eqref{C3} will be fully
analyzed here.

\section{Preliminaries}

We assume that $z k(z) \in L^2((-2,2), \CC)$ is analytic in a neighborhood of
$0$. This includes two cases: regular, when $k$ is analytic at $0$, and
singular, when $k$ has a simple pole at $0$, in which case the integral is
understood in the principal value sense. Further, assume that $L, L_j$ are second order differential operators:

\begin{equation} \label{L}
\begin{cases}
L u = \mathcal{a} u'' + \mathcal{b} u' + \mathcal{c} u,
\\
\mathcal{a}(\pm 1) = 0, \ \mathcal{b}(\pm 1) = \mathcal{a}'(\pm 1) ,
\end{cases}
\end{equation}

\noindent where the indicated boundary conditions are necessary for the above
commutation relations to hold. These are also necessary for the adjoint
operator to be a differential operator as well. Thus various classes of
operators, such as self-adjoint, symmetric or normal can be described by specifying additional constraints on the coefficients of $L$, always assuming that the boundary conditions in \eqref{L} hold.

When $k$ is smooth in $[-2,2]$, formulating commutation relations \eqref{C3}
and \eqref{C2} in terms of the kernel $k(z)$ and the coefficients of $L$ is a
matter of integration by parts, which
due to the imposed boundary conditions lead, respectively, to

\begin{equation}\label{R3}
\begin{split}
[\mathcal{a}(y+z)-\mathcal{a}(y)] k''(z) + [2\mathcal{a}'(y) + \mathcal{b}(y+z) - \mathcal{b}(y)] k'(z) +& \\
+[\mathcal{c}(y+z) - \mathcal{c}(y) + \mathcal{b}'(y) - \mathcal{a}''(y)] k(z) &=0,
\end{split}
\tag{R1}
\end{equation}

\begin{equation}\label{R2}
\begin{split}
[\mathcal{a}_2(y+z)-\mathcal{a}_1(y)] k''(z) + [2\mathcal{a}_1'(y) + \mathcal{b}_2(y+z) - \mathcal{b}_1(y)] k'(z) +& \\
+[\mathcal{c}_2(y+z) - \mathcal{c}_1(y) + \mathcal{b}_1'(y) - \mathcal{a}_1''(y)] k(z) &=0,
\end{split}
\tag{R2}
\end{equation}

\noindent where $\mathcal{a}_j, \mathcal{b}_j, \mathcal{c}_j$ denote the
coefficients of $L_j$ for $j=1,2$. Less obviously (see Remark~\ref{REM k has a
  pole}), the same relation \eqref{R3} holds if $k$ has a simple pole at $0$.

The main idea of the proofs is to analyze these relations by taking sufficient
number of derivatives in $z$ and evaluating the result at $z=0$. This allows
one to find linear differential relations between the coefficients of the
differential operators, narrowing down the set of possibilities to families of
functions depending on finitely many parameters. Returning to the original
relations \eqref{R3}, \eqref{R2} we obtain necessary and sufficient conditions
for commutation that can be completely analyzed, resulting in the explicit
listing of all pairs $(k,L)$ satisfying \eqref{R3}.

\begin{remark}
The complete analysis of \eqref{C2} beyond the instances generated by \eqref{C3}, can also be achieved by our approach, but will require substantially more work. We remark that in this case too it can be shown that either $k$ is trivial or the coefficients of $L_1$ and $L_2$ are linear combinations of polynomials multiplied by exponentials.     
\end{remark}

\section{Main Results}

\begin{mydef} \label{trivial DEF}

We will say that $k$ (or operator $K$) is \textit{trivial}, if it is a finite linear combination of exponentials $e^{\alpha z}$ or has the form  $e^{\alpha z} p(z)$, where $p(z)$ is a polynomial. Note that in this case $K$ is a finite-rank operator.

\end{mydef}

\begin{remark} \label{REM multiplier}
When $K$ commutes with $L$, then $M K M^{-1}$ commutes with $M L M^{-1}$. If $M$ is the multiplication operator by $z \mapsto e^{\tau z}$, then $M K M^{-1}$ is a finite convolution operator with kernel $k(z) e^{\tau z}$ (where $k$ is the kernel of $K$) and $M L M^{-1}$ is a second order differential operator with the same leading coefficient as $L$. With this observation the results of Theorem~\ref{THM commutation} are stated up to multiplication of $k$ by $e^{\tau z}$, i.e. we chose a convenient constant $\tau$ in order to more concisely state the results.  Moreover, one can add any complex constant to $\mathcal{c}(y)$ (cf. \eqref{L}), which corresponds to adding a multiple of identity to $L$ and hence the commutation still holds.
\end{remark}

\noindent In theorem below all parameters are complex, unless specified otherwise. 

\begin{thm}[Commutation \eqref{C3}] \label{THM commutation}
Let $K, L$ be given by \eqref{K} and \eqref{L} with $\mathcal{a}, \mathcal{b}, \mathcal{c}$ smooth in $[-2,2]$. Assume $k$ is smooth in $[-2,2] \backslash \{0\}$ and either it

\begin{enumerate}
\item[(i)] is analytic at $0$, not identically zero near $0$ and is nontrivial in the sense of Definition~\ref{trivial DEF}.

\item[(ii)] has a simple pole at $0$. 
\end{enumerate}

\noindent If \eqref{R3} holds, then (in case $\lambda$ or $\mu = 0$ appropriate limits must be taken)

\begin{equation} \label{k general}
 k(z) = \frac{\lambda }{\sinh \left( \frac{\lambda}{2} z \right)} \left( \alpha_1 \frac{\sinh(\mu z)}{\mu} + \alpha_2 \cosh(\mu z) \right)
\end{equation}

\begin{equation} \label{a,b,c general}
\begin{cases}
\mathcal{a}(y) = \frac{1}{\lambda^2} \left[ \cosh(\lambda y) - \cosh \lambda \right]
\\
\mathcal{b}(y)= \mathcal{a}'(y)
\\
\mathcal{c}(y)= \left( \tfrac{\lambda^2}{4} - \mu^2 \right) \mathcal{a}(y)
\end{cases}
\end{equation}

\noindent For some special choices of parameters, the differential operator commuting with $K$ is more general than the one given by \eqref{a,b,c general}. Below we list all such cases:

\begin{enumerate}

\item $\alpha_1 = 0, \ \lambda = \pi i, \ \mu = \frac{2m+1}{4} \lambda$ with $m \in \ZZ$:

\begin{equation*}
k(z) = \frac{\cos \left( \tfrac{\pi (2m+1)}{4}z \right)} { \sin \left( \tfrac{\pi}{2}z \right)}
\qquad \text{and} \qquad
\begin{cases}
\mathcal{a}(y) = \alpha \left( e^{\pi i y} - e^{\pi i} \right) + \beta \left( e^{-\pi i y} - e^{-\pi i} \right)
\\
\mathcal{b}(y)= \mathcal{a}'(y)
\\
\mathcal{c}(y)= \frac{\pi^2}{4} \left[ \frac{(2m+1)^2}{4} - 1 \right] \mathcal{a}(y)
\end{cases}
\end{equation*}

\noindent When $\alpha = \beta$ \eqref{a,b,c general} is recovered.

\item $\alpha_1 = \mu =0$, then with $\mathcal{a}_0(y)=\cosh(\lambda y) - \cosh\lambda$:

\begin{equation*}
k(z) = \frac{1}{\sinh \left( \tfrac{\lambda}{2} z \right)}
\qquad \text{and} \qquad
\begin{cases}
\mathcal{a} (y) = \alpha \mathcal{a}_0(y)
\\
\mathcal{b}(y)=\alpha \mathcal{a}_0'(y) + \beta \mathcal{a}_0(y)
\\
\mathcal{c}(y)= \frac{\beta}{2} \mathcal{a}_0'(y) + \alpha \frac{\lambda^2}{4} \mathcal{a}_0(y)
\end{cases}
\end{equation*}

\noindent When $\beta = 0$ \eqref{a,b,c general} is recovered.

\item $\mu = \lambda = 0$, then with $\mathcal{p}(y)$ an arbitrary polynomial of order at most two such that $\mathcal{p}'(0)=0$:

\begin{equation*}
k(z) = \frac{1}{\beta} + \frac{1}{z}
\qquad \text{and} \qquad
\begin{cases}
\mathcal{a}(y)=(y^2-1)\mathcal{p}(y) 
\\
\mathcal{b}(y)=\mathcal{a}'(y) + \beta y \mathcal{p}'(y) - \beta \mathcal{p}''(y)
\\
\mathcal{c}(y)=\beta \mathcal{p}'(y)
\end{cases}
\end{equation*}

\noindent When $\mathcal{p}(y) \equiv 1$ \eqref{a,b,c general} is recovered.

\item $\mu = \lambda = \alpha_1 = 0$, then with $\mathcal{p}(y)$ an arbitrary polynomial of order at most two:

\begin{equation*}
k(z) = \frac{1}{z}
\qquad \text{and} \qquad
\begin{cases}
\mathcal{a}(y)=(y^2-1)\mathcal{p}(y) 
\\
\mathcal{b}(y)=a'(y) + \beta (y^2-1)
\\
\mathcal{c}(y)=y \mathcal{p}'(y) + \beta y
\end{cases}
\end{equation*}

\noindent When $\mathcal{p}(y) \equiv 1$ and $\beta = 0$ \eqref{a,b,c general} is recovered.

\end{enumerate}

\end{thm}

\begin{remark} \label{REM lambda in iR}
If $\lambda \in i\RR$, then $k(z)$ may become singular at $z \in [-2,2]\backslash \{0\}$. In order to exclude these cases we need to require either 

\begin{enumerate}
\item[$\bullet$] $|\lambda| < \pi$

\item[$\bullet$] $\pi \leq |\lambda| < 2 \pi$ and $\alpha_1 = 0, \ \mu = \lambda \frac{2m+1}{4}$ for some $m \in \ZZ$
\end{enumerate}
\end{remark}

\begin{remark} \mbox{} \label{REM morrison}
\begin{enumerate} 

\item[(i)] Morrison's result corresponds to the analytic case: $\alpha_2 = 0$ and when $k$ is even and real-valued. According to Remark~\ref{REM multiplier} the general integral operator in the analytic case is similar to Morrison's operator and therefore its spectrum can be determined using Morrison's results.

\item[(ii)] In Theorem \ref{THM commutation} $k$, as well as $L$, can independently be multiplied by arbitrary complex constants, which we sometimes omit to achieve a simpler form of $k$ and $L$.

\end{enumerate}

\end{remark}

\begin{remark}
As we have already mentioned, the connections between the coefficient functions of the differential operators are obtained by differentiating the relation \eqref{R3} appropriate number of times and setting $z=0$. Smoothness of coefficients, analyticity of $k$ at zero (the fact that $k$ is nontrivial and that it doesn't vanish near $0$) are used at this stage, to argue that the differentiation procedure can be terminated at some point and the connections between the coefficient functions will follow. Thus, the original assumptions can be replaced by requiring appropriate degree of smoothness on $k$ and the coefficient functions and that some expression(s) involving $k^{(j)}(0)$ is not zero. This expression can be easily found from our analysis. For example the hypotheses of Theorem~\ref{THM commutation} (case $(i)$) can be replaced by $\mathcal{a},\mathcal{b},\mathcal{c},k \in C^3$ and $k^2(0) k''(0) - k(0) k'(0) \neq 0$ (cf. Section~\ref{Comm analytic SECTION}). Analogous changes can be made in case $(ii)$ of Theorem~\ref{THM commutation}.
\end{remark}

\begin{remark} \label{REM k has a pole}
When $k$ has a pole at zero, the commutation is understood in the principal value sense, namely

\begin{equation*}
\lim_{\epsilon \to 0} \int_{[-1,1] \backslash B_\epsilon (x)} k(x-y) Lu(y) \D y - L \int_{[-1,1] \backslash B_\epsilon (x)} k(x-y) u(y) \D y = 0 .
\end{equation*} 

\noindent After integrating by parts, this can be rewritten as

\begin{equation*}
\lim_{\epsilon \to 0} \int_{[-1,1] \backslash B_\epsilon (x)} F(x,y) u(y) \D y + \Phi(u, x, \epsilon) = 0 ,
\end{equation*}

\noindent where $F(x,y)$ is the left-hand side of \eqref{R3} with $z = x-y$ and

\begin{equation*}
\begin{split}
\Phi(u, x, \epsilon) =& k(\epsilon) \Big\{ \big[ \mathcal{a}(x-\epsilon) - \mathcal{a}(x) \big] u'(x-\epsilon) + \big[ \mathcal{b}(x-\epsilon) - \mathcal{b}(x) - \mathcal{a}'(x-\epsilon)\big] u(x-\epsilon)  \Big\} -
\\
-& k(-\epsilon) \Big\{ \big[ \mathcal{a}(x+\epsilon) - \mathcal{a}(x) \big] u'(x+\epsilon) + \big[ \mathcal{b}(x+\epsilon) - \mathcal{b}(x) - \mathcal{a}'(x+\epsilon) \big] u(x+\epsilon)  \Big\} +
\\
+& k'(\epsilon) u(x-\epsilon) \big[ \mathcal{a}(x-\epsilon) - \mathcal{a}(x) \big] - k'(-\epsilon) u(x+\epsilon) \big[ \mathcal{a}(x+\epsilon) - \mathcal{a}(x) \big] .
\end{split}
\end{equation*}

\noindent Expanding $\Phi(u, x, \epsilon)$ in $\epsilon$ we observe that all terms
up to $O(\epsilon)$ cancel out and hence, $\lim_{\epsilon \to 0} \Phi(u, x,
\epsilon) = 0$. Therefore we conclude $F(x,y) = 0$ for $y \neq x$, resulting
in the same relation \eqref{R3}, as in smooth case.

\end{remark}

\begin{remark} \label{normal}

\noindent As was discussed in the introduction one might want to check whether $L$ (given by \eqref{L}) is normal: $L L^* = L^* L$. Recall that

\begin{equation*}
L^* u = \overline{\mathcal{a}} u'' + (2 \overline{\mathcal{a}}' - \overline{\mathcal{b}}) u' + (\overline{\mathcal{a}}'' - \overline{\mathcal{b}}' + \overline{\mathcal{c}}) u ,
\end{equation*}

\noindent therefore we find

\begin{equation*}
L = L^* \quad \Longleftrightarrow \quad \Im \mathcal{a} = 0, \quad \Re \mathcal{b} = \mathcal{a}' \quad \text{and} \quad \Im \mathcal{c} = \tfrac{1}{2} \Im \mathcal{b}' .
\end{equation*}

\noindent To analyze the normality relation, we first give the conditions for commutation of $L$ with another differential operator $D u = \mathcal{A} u'' + \mathcal{B} u' + \mathcal{C} u$, assuming $\mathcal{a} \neq 0$. One can find that

\begin{equation*}
\begin{split}
LDu &= \mathcal{a} \mathcal{A} u^{(4)} + \left[\mathcal{a}(2\mathcal{A}' + \mathcal{B}) + \mathcal{b} \mathcal{A}\right] u^{(3)} + \left[ \mathcal{a} (\mathcal{A}'' + 2 \mathcal{B}' + \mathcal{C}) + \mathcal{b} (\mathcal{A}' + \mathcal{B}) + \mathcal{c} \mathcal{A} \right] u'' +
\\
&+ \left[ \mathcal{a} (\mathcal{B}'' + 2 \mathcal{C}') + \mathcal{b} (\mathcal{B}' + \mathcal{C}) + \mathcal{c} \mathcal{B} \right] u'
+ \left[ \mathcal{a} \mathcal{C}'' + \mathcal{b} \mathcal{C}' + \mathcal{c} \mathcal{C} \right] u .
\end{split}
\end{equation*}

\noindent Comparing this with an analogous expression for $DLu$ and equating the coefficients of corresponding derivatives of $u$ we obtain that $LD = DL$ is equivalent to

\begin{equation} \label{LD = DL}
\begin{cases}
\mathcal{a} \mathcal{A}' = \mathcal{A} \mathcal{a}'
\\
2 \mathcal{a} \mathcal{B}' + \mathcal{b} \mathcal{A}' =  2 \mathcal{A} \mathcal{b}' + \mathcal{B} \mathcal{a}'
\\
\mathcal{a} \mathcal{B}'' + 2 \mathcal{a} \mathcal{C}' + \mathcal{b} \mathcal{B}' = \mathcal{A} \mathcal{b}'' + 2 \mathcal{A} \mathcal{c}' + \mathcal{B} \mathcal{b}'
\\
\mathcal{a} \mathcal{C}'' + \mathcal{b} \mathcal{C}' = \mathcal{A} \mathcal{c}'' + \mathcal{B} \mathcal{c}'
\end{cases}
\end{equation}

\noindent The first equation of \eqref{LD = DL} implies $\mathcal{A} = \alpha \mathcal{a}$ for some $\alpha \in \CC$. Using this in the second equation of \eqref{LD = DL} we get $\beta \mathcal{a} = (\mathcal{B} - \alpha \mathcal{b})^2$ for some $\beta \in \CC$. The third relation reads

\begin{equation*}
\begin{split}
\mathcal{C}' &= \alpha \mathcal{c}' - \tfrac{1}{2} (\mathcal{B}'' - \alpha \mathcal{b}'') + \frac{\mathcal{B} \mathcal{b}' - \mathcal{b} \mathcal{B}'}{2 \mathcal{a}} =
\\
&= \alpha \mathcal{c}' + \frac{\beta}{2} \frac{\left( \mathcal{b}' - \tfrac{\mathcal{a}''}{2} \right) (\mathcal{B}- \alpha\mathcal{b}) - \left( \mathcal{b} - \tfrac{\mathcal{a}'}{2} \right) (\mathcal{B}'- \alpha\mathcal{b}')}{(\mathcal{B}- \alpha\mathcal{b})^2} ,
\end{split}
\end{equation*} 

\noindent where in the last step we used the identity $2\mathcal{a}(\mathcal{B}''- \alpha\mathcal{b}'') = \mathcal{a}'' (\mathcal{B}- \alpha\mathcal{b}) - \mathcal{a}' (\mathcal{B}'- \alpha\mathcal{b}')$. Integrating, we find $\mathcal{C} = \alpha \mathcal{c} + \frac{1}{2} f + \text{const}$, where 

\begin{equation*}
f = \frac{\beta}{2} \ \frac{2\mathcal{b} - \mathcal{a}'}{\mathcal{B} - \alpha \mathcal{b}} .
\end{equation*}

\noindent When $\mathcal{B} = \alpha \mathcal{b}$, then $\beta = 0$ and by convention we assume $f = 0$. Finally, substituting the expression for $\mathcal{C}$, the fourth equation of \eqref{LD = DL} can be simplified to

\begin{equation*}
2 \beta \mathcal{c}' = (\mathcal{B} - \alpha \mathcal{b}) f'' + \frac{\beta \mathcal{b}}{\mathcal{B} - \alpha \mathcal{b}} f' =  \left[(\mathcal{B} - \alpha \mathcal{b}) f' \right]' + f f' .
\end{equation*}

\noindent Now we integrate the last relation and putting everything together we conclude

\begin{equation*}
LD = DL \quad \Longleftrightarrow \quad
\begin{cases}
\mathcal{A} &= \alpha \mathcal{a},
\\
\beta \mathcal{a} &= (\mathcal{B} - \alpha \mathcal{b})^2,
\\
\mathcal{C} &= \alpha \mathcal{c} + \frac{1}{2} f + \text{const},
\\
2 \beta \mathcal{c} &= (\mathcal{B} - \alpha \mathcal{b}) f' + \frac{1}{2} f^2 + \text{const}.
\end{cases}
\end{equation*}

Write $L = L_0 + L_1$, where $2L_0 = L + L_*$ is self-adjoint and $2L_1 = L - L_*$ is skew-adjoint. Clearly $L$ is normal, \IFF $L_0$ commutes with $L_1$. The coefficient of $\frac{\D^2}{\D x^2}$ in $L_0$ is $\Re \mathcal{a}$ and in $L_1$ is $i \Im \mathcal{a}$. The first equation for commutation of $L_0, L_1$ implies $\Im \mathcal{a} = \alpha \Re \mathcal{a}$ for some $\alpha \in \RR$. W.l.o.g. we may take $\alpha = 0$. Indeed, $L$ is normal \IFF $\tilde{L}=(1-i\alpha) L$ is normal. Now the coefficient of $\frac{\D^2}{\D x^2}$ in $\tilde{L}_1$ is $\frac{1}{2} [(1-i\alpha) \mathcal{a} - (1+i\alpha) \overline{\mathcal{a}}] = 0$. Thus, w.l.o.g. $L=L_0 + L_1$ where $L_0$ is a second order self-adjoint operator and $L_1$ is of first order and skew-adjoint.  Simplifying commutation relations for $L_0, L_1$ we find

\begin{equation*}
\begin{split}
L L^* = L^* L \quad &\text{and} \quad L \neq L^*, \qquad \text{iff}
\\[.2in]
\begin{cases}
L = L_0 + \gamma L_1, \quad \gamma \in \RR \backslash \{0\}, \\
L_0u = \mathcal{a} u''+ \mathcal{b}_0 u' + \mathcal{c}_0 u, \\
L_1u = \mathcal{b}_1u' + \mathcal{c}_1 u,
\end{cases}
\quad &\text{and} \quad
\begin{cases}
\mathcal{a} \in \RR \quad \text{and w.l.o.g.} \ \mathcal{a}>0,
\\
\mathcal{b}_1 = \sqrt{\mathcal{a}},
\\[.1in]
\mathcal{c}_1 = \dfrac{2 \mathcal{b}_0 - \mathcal{a}'}{\sqrt{\mathcal{a}}} + i\RR,
\\
\Re \mathcal{b}_0 = \mathcal{a}',
\\
4\mathcal{c}_0 = 2 \mathcal{b}_0' - \mathcal{a}'' + \dfrac{(\mathcal{a}'-2 \mathcal{b}_0)(3\mathcal{a}' - 2\mathcal{b}_0)}{2\mathcal{a}} + \RR.
\end{cases}
\end{split}
\end{equation*}

\noindent The listed conditions in particular imply that $L_0$ is self adjoint and $L_1$ is skew-adjoint. 
\end{remark}

\vspace{.2in}

Theorem~\ref{THM commutation} characterizes solutions of the commutation relation $K L u= L K u$, where $u$ is a smooth function on $[-1,1]$. Up to this point we were assuming that $K: L^2(-1,1) \mapsto L^2(-1,1)$, but following \cite{al-aifari katsevich}, \cite{katsevich1}, \cite{katsevich2}, \cite{katsevich tovbis} we can consider $K$ as an operator $K : L^2(-1,1) \mapsto L^2(a,b)$ by restricting the variable $x$ in $(Ku)(x)$ to $(a,b)$, where $(a,b)$ is the line segment connecting $a$ to $b$ in the complex plane. Now let $L_2 : = L_{(a,b)}$ denote the operator $L$ acting on (and returning) functions defined on the line segment $(a,b)$ and similarly $L_1 := L_{(-1,1)}$. If both $L_1$ and $L_2$ are self-adjoint (in particular we need the coefficient of $\frac{\D^2}{\D y^2}$ in $L$ to vanish at $\pm 1, a$ and $b$) we get an example of commutation \eqref{C2}: $K L_1 u = L_2 K u$, where $u$ is a smooth function on $[a,b]$. Below we present all such instances that can be deduced from the commutation relation $KL = LK$ (the results are given up to multiplication of $k(z)$ by $e^{\tau z}$, cf Remark~\ref{REM adding tau in examples} below).

\begin{corollary} \label{CORO examples}

Let $K : L^2(-1,1) \to L^2(a,b)$ be given by \eqref{K} and $L$ be a differential operator given by \eqref{L}, then the commutation relation

\begin{equation} \label{shifted commutation}
\begin{cases}
K L_{(-1,1)} u = L_{(a,b)} K u  \qquad \qquad u \in C^\infty [-1,1]
\\
L_{(-1,1)}^* = L_{(-1,1)} \quad \text{and} \quad L_{(a,b)}^* = L_{(a,b)}
\end{cases}
\end{equation}

\noindent holds for the following choices of operators $K,L$ and line segments $(a,b)$:

\begin{enumerate}

\item $k$ is given by \eqref{k general}, coefficients of $L$ are given by \eqref{a,b,c general} with

\begin{equation*}
\begin{cases}
\lambda , \mu \in \RR \cup i\RR, \quad \lambda \neq 0
\\
a = -1 + \frac{2\pi i n}{\lambda}, \quad b = 1 + \frac{2\pi i n}{\lambda} \quad, \ n \in \ZZ
\end{cases}
\end{equation*}

\noindent (When $\lambda \in i\RR$ further restrictions of Remark~\ref{REM lambda in iR} must be taken into account)

\item $k(z) = \dfrac{1}{\sinh \left( \tfrac{\lambda}{2} z \right)}$ and with $\mathcal{a}_0(y)=\cosh(\lambda y) - \cosh\lambda$:

\begin{equation*}
\begin{cases}
\mathcal{a} (y) = \alpha \mathcal{a}_0(y)
\\
\mathcal{b}(y)=\alpha \mathcal{a}_0'(y) + \beta \mathcal{a}_0(y)
\\
\mathcal{c}(y)= \frac{\beta}{2} \mathcal{a}_0'(y) + \alpha \frac{\lambda^2}{4} \mathcal{a}_0(y) ,
\end{cases} 
\end{equation*}

where $\beta \in i\RR$, \ $\lambda \in \RR \cup i\RR$, \ $\alpha \in \RR$ and $a = -1 + \frac{2\pi i n}{\lambda}, \quad b = 1 + \frac{2\pi i n}{\lambda}$ with $\ n \in \ZZ$.

\item $k(z) = \displaystyle \frac{1}{\beta} + \frac{1}{z}$ and $L$ has coefficients

\begin{equation*}
\begin{cases}
\mathcal{a}(y) = (y^2-1)(y^2-b^2) 
\\[.1in]
\mathcal{b}(y) = \mathcal{a}'(y) + 2\beta (y^2-1)
\\[.1in]
\mathcal{c}(y) = 2 \beta y  ,
\end{cases}
\end{equation*}

\noindent where $\beta \in i\RR$, \ $a = -b$ and $b>0$.

\item $k(z) = \displaystyle \frac{1}{z}$ and $L$ has coefficients

\begin{equation*}
\begin{cases}
\mathcal{a}(y) = (y^2-1)(y-a) (y-b) 
\\[.1in]
\mathcal{b}(y) = \mathcal{a}'(y) + \beta (y^2-1)
\\[.1in]
\mathcal{c}(y) = 2 y^2 + \left(\beta -a-b\right) y  ,
\end{cases}
\end{equation*}

\noindent where $\beta \in i\RR$ and $a<b$ are real.

\end{enumerate}

\end{corollary}

\begin{proof}
The proof immediately follows from Theorem~\ref{THM commutation} and discussion above, we just mention that in item 1 the restrictions $\lambda, \mu \in \RR \cup i\RR$ make $L$ self-adjoint on $[-1,1]$, the choice of $a,b$ follows from the fact that coefficients of $L$ are $\frac{2\pi i}{\lambda}$-periodic. Therefore, $L$ is also self-adjoint on $[a,b]$. Similarly, in items 2, 3 and 4 the condition $\beta \in i\RR$ guarantees self-adjointness of $L$. In item 3 we are forced to take $a = -b$, because in the corresponding commutation relation (item 3 of Theorem~\ref{THM commutation}) $\mathcal{a}(y) = (y^2-1)\mathcal{p}(y)$ where $\mathcal{p}'(0) = 0$, hence $\mathcal{p}(y) = y^2 - b^2$.
\end{proof}

\vspace{.1in}

\begin{remark} \label{REM adding tau in examples}
Due to Remark~\ref{REM multiplier} it is easy to check that in Corollary~\ref{CORO examples}, in each of the four items $K$ can be replaced by $MKM^{-1}$ and $L$ by $MLM^{-1}$, where $M$ is multiplication operator by $e^{\tau z}$ and (in addition to given parameter restrictions) it must hold $\tau \in i\RR$ in order for $MLM^{-1}$ to be self-adjoint. Note that in this case $M$ is a unitary operator, therefore $M L M^{-1}$ is self-adjoint \IFF $L$ is.  However, for item 2 there is an additional case: $\tau \in \CC$ and $\beta = 2 i \alpha \Im \tau$.
\end{remark}

\begin{remark}
Taking $\beta = 0$ in item 4 we obtain the commutation used in \cite{al-aifari katsevich}, \cite{katsevich1}, \cite{katsevich2}, \cite{katsevich tovbis} mentioned in the introduction. Indeed, since any real constant can be added to $\mathcal{c}$ we can rewrite $\mathcal{c}(y) = 2 \left( y - \frac{a+b}{4} \right)^2$, which is precisely the form of $\mathcal{c}$ used in those references.  
\end{remark}

\begin{remark} \label{REM discrete spectrum}
Observe that in all of the cases $k(z)$ has a singularity and the
corresponding operator $K$ is not compact. The spectrum of $K^{*}K$ therefore,
need not be discrete (e.g. \cite{koppelman pincus}). Yet it was found to be
discreet in most cases of the finite Hilbert transform SVD
\cite{al-aifari katsevich,katsevich1,katsevich2,katsevich tovbis}. The discreteness of the SVD decomposition comes from the
discreteness of the spectrum of self-adjoint differential operators $L_{1}$ and $L_{2}$ in \eqref{C2}, provided that singularities of $Ku$ are not at the
end-points of the interval for the Sturm-Liouville eigenvalue problem for
$L_{2}$. In particular the situation when $(-1,1)$ and $(a,b)$ intersect does not in and of itself cause the appearance of continuous spectrum. In the context of operators listed in Corollary~\ref{CORO examples} we can characterize when true singularities occur.
Let $\{z_j\}$ be the simple poles of $k$, then the function $(Ku)(\xi) =
\int_{-1}^1 k(\xi-y) u(y) \D y$ may have (logarithmic) singularities at $\{z_j
\pm 1\}$ (cf. \cite{gakhov} sections 8.5 and 8.5). Let also $\{y_j\}$ be the
zeros of $\mathcal{a}(y)$. If the set of removable singularities $\{y_j\} \setminus \{z_j \pm 1\}$  has
at least two points, say $a$ and $b$, then $K u$ is regular at points $a, b$
and so (using \eqref{shifted commutation}) $K$ maps eigenfunctions of
$L_{(-1,1)}$ to eigenfunctions of $L_{(a,b)}$, making the former the
eigenfunctions of $K^{*}K$. We will call this case regular. Generically, all
operators in items 1 and 2 from Corollary~\ref{CORO examples} belong to the singular
case. Regular cases arise for special choices of parameters, for which some of
the singularities of $k(z)$ are eliminated. For example, taking 
$\alpha_1 = 0, \ \lambda = i \frac{\pi}{2}, \ \mu = i \frac{\pi}{8}$ we obtain

\begin{equation*}
k(z) = \frac{1}{\sin \left( \frac{\pi}{8}z \right)}, \qquad \qquad
\begin{cases}
\mathcal{a}(y) = \cos \left( \frac{\pi}{2}y \right)
\\
\mathcal{b}(y) = \mathcal{a}'(y) 
\\
\mathcal{c}(y) = - \frac{3\pi^2}{64} \mathcal{a}(y)
\end{cases} .
\end{equation*}

\noindent Now the set of removable singularities is $\{1+2n\}_{n \in \ZZ} \backslash \{8m \pm 1\}_{m \in
  \ZZ}$, which contains the points $a = 3, \ b=5$.
\end{remark}

\section{Commutation, regular case} \label{Comm analytic SECTION}

\begin{lemma} \label{LEM a b c comm}
Assume the setting of Theorem~\ref{THM commutation} case $(i)$, then for some complex constants $\alpha, \nu$ we have

\begin{equation} \label{a ode comm}
\mathcal{a}'''(y) + \alpha \mathcal{a}(y) = 0, \qquad \mathcal{b}(y) = \mathcal{a}'(y), \qquad \mathcal{c}(y) = \nu \mathcal{a}(y) .
\end{equation}

\end{lemma}

\begin{proof}
Write $k(z) = \sum_{n=0}^\infty \frac{k_n}{n!} z^n$ near $z=0$. The $n$-th derivative of \eqref{R3} w.r.t. $z$ evaluated at $z=0$ reads

\begin{equation} \label{nth derivative of comm relation}
\begin{split}
2\mathcal{a}'(y) k_{n+1} + [\mathcal{b}'(y)-\mathcal{a}''(y)] k_{n} + \sum_{j=0}^{n-1} C_j^n \mathcal{a}^{(n-j)}(y) k_{j+2} +& 
\\
+\sum_{j=0}^{n-1} C_j^n \mathcal{b}^{(n-j)}(y) k_{j+1} + 
\sum_{j=0}^{n-1} C_j^n \mathcal{c}^{(n-j)}(y) k_{j} &= 0 ,
\end{split}
\end{equation}

\noindent where $C_j^n  = {n \choose j}$. The above relation for $n=0$ gives

\begin{equation} \label{n=0 comm relation}
2 k_1 \mathcal{a}'(y) + [\mathcal{b}'(y) - \mathcal{a} ''(y)] k_0 = 0 .
\end{equation}

Assume first $k_0 = 0$, then $k_1 = 0$ (otherwise the boundary conditions imply $\mathcal{a}=0$). By induction one can conclude $k_j = 0$ for any $j$. Indeed, let $k_j = 0$ for $j=0,...,n$, then \eqref{nth derivative of comm relation} reads 

\begin{equation*}
(n+2) \mathcal{a}'(y) k_{n+1} = 0 .
\end{equation*}

\noindent Hence the boundary conditions imply $k_{n+1} = 0$. So if $k_0 = 0$, then $k(z)$ must be identically zero near $z=0$, which we
do not allow. 

Thus $k_0 \neq 0$, and in view of Remark~\ref{REM multiplier} we may assume $k_1 = k'(0) = 0$ (otherwise multiply $k(z)$ by $e^{-k_1/k_0 z}$). Taking into account the boundary conditions, from \eqref{n=0 comm relation} we obtain $\mathcal{b}(y) = \mathcal{a}'(y)$. Now we substitute this in \eqref{nth derivative of comm relation} with $n=1$, integrate the result to find the expression for $\mathcal{c}$ in \eqref{a ode comm} with $\nu = - \tfrac{3k_2}{k_0}$. When $n=2$ equation \eqref{nth derivative of comm relation}, after
elimination of $\mathcal{b}$ and $\mathcal{c}$ becomes $k_3 \mathcal{a}'(y) = 0$
and we conclude that $k_3 = 0$.
When $n = 3$, we find

\[
k_0 k_2 \mathcal{a}'''(y) + (5k_0 k_4 - 9 k_2^2) \mathcal{a}'(y) = 0 .
\]

If $k_2 = 0$, then $k_4 = 0$ and as can be immediately seen from \eqref{nth derivative of comm relation}, induction argument shows that $k_j = 0$ for all $j\ge 1$. 
Thus, we may assume $k_2 \neq 0$, in which case $\mathcal{a}$ satisfies the ODE in \eqref{a ode comm}.
\end{proof}

\vspace{.1in}

\noindent From \eqref{a ode comm} $\mathcal{a}$ has to have one of the following forms, with $a_j \in \CC$

\begin{enumerate}
\item[I.] $\displaystyle \mathcal{a}(y) = a_1 e^{\lambda y} + a_2 e^{- \lambda y} + a_0 $, with $0 \neq \lambda \in \CC$

\item[II.] $\displaystyle \mathcal{a}(y) = a_2 y^2 + a_1 y + a_0 $ 
\end{enumerate}

\vspace{.1in}

\noindent $\bullet$ Assume case I holds, replacing the expressions for $\mathcal{a},\mathcal{b},\mathcal{c}$ from Lemma~\ref{LEM a b c comm}, \eqref{R3} becomes a linear combination of exponentials $e^{\pm \lambda y}$ with coefficients depending only on $z$, hence each coefficient  must vanish. These can be simplified as $a_j \left\{ k'' +  \lambda \coth\left( \frac{\lambda}{2} z \right) k' + \nu k \right\} = 0$ for $j=1,2$. Of course, at least one of $a_1, a_2$ is different from zero and so we deduce

\begin{equation} \label{comm analytic k exp ODE}
\displaystyle k'' + \lambda \coth\left( \tfrac{\lambda}{2} z \right) k' + \nu k = 0 .
\end{equation}  

\noindent Setting $u(z) = k(z) \sinh \left( \tfrac{\lambda}{2} z \right)$,
the above ODE becomes $u'' + \left( \nu - \frac{\lambda^2}{4}
\right) u = 0$. So,

\begin{equation*}
k(z) = \frac{\sinh (\mu z)}{\mu \sinh \left( \tfrac{\lambda}{2} z \right)}
\qquad \qquad \mu^2 = \frac{\lambda^2}{4} - \nu .
\end{equation*}

\noindent When $\mu = 0$, the formula is understood in the limiting sense. Note that this is \eqref{k general} with $\alpha_2 = 0$ (here $\alpha_2$ refers to the
parameter in formula \eqref{k general}, whose vanishing makes $k(z)$ analytic
on $[-2,2]$.) Because $\mathcal{a}(y)$ satisfies the boundary conditions we must have $a_1 = a_2$ or $\lambda \in \pi i n$ for some $n \in \ZZ$. If $\lambda = \pi i n$, then for $k$ to be smooth in $[-2,2]$ we must have $\mu \neq 0$, moreover $\sinh \left( \tfrac{2\mu m}{n} \right) = 0$ for any $m \in \ZZ$ with $\frac{m}{n} \in [-1,1]$. In particular this should hold for $m=1$, which implies $\mu = \frac{\lambda l}{2}$ for some $l \in \ZZ$, which in turn implies that $k$ is a trigonometric polynomial, and hence is trivial. Thus we may assume $\lambda \notin \pi i \ZZ$, and so $a_1 = a_2$, showing that $\mathcal{a}(y) = \cosh(\lambda y) - \cosh \lambda$. 

Now we show that if $\lambda \in i\RR$, then it must hold $|\lambda| < \pi$. Otherwise, $k$ is trivial. Indeed, assume $\lambda \in i\RR$ and $|\lambda| \geq \pi$ we see that the denominator of $k(z)$ has additional zeros at $z=\pm \frac{2 \pi i}{\lambda} \in [-2,2]$. In order for $k$ to be smooth, we require that its numerator also vanishes at these points. So $\sinh \left( \frac{2 \pi i}{\lambda} \mu \right) = 0$ and hence $\mu = \frac{\lambda}{2} m$ for some $m \in \ZZ$. But then, again $k$ is a trigonometric polynomial.

\vspace{.1in}

\noindent $\bullet$ Assume case II holds, then $\mathcal{a}(y) = a_2 (y^2 - 1)$ and substituting into \eqref{R3} we find

\begin{equation}  \label{comm analytic k pol ODE}
z k'' + 2 k' + \nu z k = 0 .
\end{equation}

\noindent Setting $u(z) = z k(z)$ the ODE turns into $u'' + \nu u = 0$, which corresponds to the limiting case $\lambda = 0$ in the formulas for $k$ and $\mathcal{a}$ and concludes the proof of Theorem~\ref{THM commutation} case $(i)$.

\section{Commutation, singular case}

Here we prove Theorems~\ref{THM commutation} case $(ii)$. In the first subsection below we obtain the possible forms for the functions $\mathcal{a},\mathcal{b}$ and $\mathcal{c}$. In the second one we do reduction of these forms, and finally in the third one we find $k$.

\subsection{Forms of $\mathcal{a} ,\mathcal{b}$ and $\mathcal{c}$}

By the assumption $k(z) = z^{-1} (k_0 + k_1 z +...)$, with $k_0 \neq 0$. So by rescaling we let $k_0 = 1$ and in view of Remark~\ref{REM multiplier} we may assume $k_1 = 0$ (otherwise multiply $k(z)$ by $e^{-k_1/k_0 z}$). Multiply \eqref{R3} by $z^3$ and refer to the resulting relation by (E). Differentiate (E) three times w.r.t. $z$ and let $z=0$ to get

\begin{equation} \label{c in terms of b singular comm*}
\mathcal{c}(y) = -\tfrac{1}{3} \mathcal{a}''(y) - 2k_2 \mathcal{a}(y) + \tfrac{1}{2} \mathcal{b}'(y) + \text{const} .
\end{equation} 

\noindent Substitute this into (E), differentiate the result 4 times w.r.t. $z$ and let $z = 0$, then

\begin{equation} \label{b'''}
\mathcal{b}''' =  \mathcal{a}^{(4)} + 24 k_2 \mathcal{a}'' - 72 k_3 \mathcal{a}' -24 k_2 \mathcal{b}' .
\end{equation}

\noindent In the fifth derivative of (E) we replace $b^{(4)}$ and $b'''$ using the above relation, then the result reads

\begin{equation} \label{alpha3 b'}
\alpha_1 \mathcal{b}' = \mathcal{a}^{(5)} + 120 k_2 \mathcal{a}^{(3)} + \alpha_1\mathcal{a}'' + \alpha_2 \mathcal{a}' ,
\end{equation}

\noindent where $\alpha_1 = - 1080 k_3$ and the
expression for $\alpha_2$ is not important. Now if $\alpha_1 = 0$ we got a
linear constant coefficient ODE for $\mathcal{a}$, otherwise we substitute the
formula for $\mathcal{b}'$ from \eqref{alpha3 b'} into \eqref{b'''} and again obtain an ODE for $\mathcal{a}$, more precisely, for some constants $\beta_j \in \CC$, either

\begin{enumerate}
\item[(A)] $\alpha_1 = 0$ and $\mathcal{a}^{(4)} + \beta_1 \mathcal{a}'' + \beta_2 \mathcal{a} = \beta_0$, or

\item[(B)] $\alpha_1 \neq 0$ and $\mathcal{a}^{(6)} + \beta_3 \mathcal{a}^{(4)} + \beta_1 \mathcal{a}'' + \beta_2 \mathcal{a} = \beta_0$

\end{enumerate} 

\noindent Therefore, using the fact that ODEs in (A) and (B) contain only even
derivatives of $\mathcal{a}$, we can conclude that in either case $\mathcal{a}$ has one of the following forms, with $p_j, a_j, \tilde{a}_j \in \CC$; \ $\lambda_j, \lambda, \mu \in \CC \backslash \{0\}$ and $\lambda \neq \pm \mu$ and $\lambda_j \neq \pm \lambda_l$ for $j \neq l$,

\begin{enumerate}
\item[I.] 

\begin{enumerate}
\item[1)] $\displaystyle \mathcal{a}(y) = \sum_{j=1}^3 (a_j e^{\lambda_j y} + \tilde{a}_j e^{-\lambda_j y}) + a_0$

\item[2)] $\displaystyle \mathcal{a}(y) =  \sum_{j=1}^2 (a_j e^{\lambda_j y} + \tilde{a}_j e^{-\lambda_j y})+\sum_{j=0}^2 p_j y^j$

\item[3)] $\displaystyle \mathcal{a}(y) =  a_1 e^{\lambda y} + \tilde{a}_1 e^{-\lambda y} + \sum_{j=0}^4 p_j y^j$
\end{enumerate}

\item[II.] 
\begin{enumerate}

\item[1)] $\displaystyle \mathcal{a}(y) =  (a_1 y+\tilde{a}_1) e^{\lambda y} + (a_2y+\tilde{a}_2) e^{-\lambda y} + a_3 e^{\mu y} + \tilde{a}_3 e^{-\mu y} + a_0$

\item[2)] $\displaystyle \mathcal{a}(y) =  (a_1 y+\tilde{a}_1) e^{\lambda y} + (a_2y+\tilde{a}_2) e^{-\lambda y} + p_2 y^2 + p_1y + p_0$

\end{enumerate}

\item[III.] $\displaystyle \mathcal{a}(y) =  (a_2 y^2 +a_1y + a_0) e^{\lambda y} + (\tilde{a}_2 y^2 +\tilde{a}_1y + \tilde{a}_0) e^{-\lambda y} + a_3$

\item[IV.] $\displaystyle \mathcal{a}(y) =  \sum_{j=0}^6 a_j y^j$
\end{enumerate}

\noindent  If $\alpha_1 \neq 0$, then from \eqref{alpha3 b'} we see that $\mathcal{b}$ has exactly the same form as $\mathcal{a}$. Assume now $\alpha_1 = 0$, if $k_2 = 0$ we find from \eqref{b'''} that $\mathcal{b}(y) = \mathcal{a}'(y) + p_2 (y^2-1)$, if $k_2 \neq 0$, then $\mathcal{b}$ is of the same form as $\mathcal{a}$ only it might contain two extra exponentials $e^{\pm \sqrt{-24k_2} y}$, if those differ from all the exponentials appearing in $\mathcal{a}$, otherwise if one of them coincides, say with $e^{\lambda y}$, then the polynomial multiplying the latter gets one degree higher. Finally, $\mathcal{c}$ is of the same form as $\mathcal{b}$.

\subsection{Reduction}

Our goal is to reduce the cases I--IV and conclude that $\mathcal{a}(y)$ can have one of the two forms $a_1 e^{\lambda y} + a_2 e^{-\lambda y} + a_0$ or $\sum_{j=0}^6 a_j y^j$. Moreover, $\mathcal{b}$ and $\mathcal{c}$ must have exactly the same form as $\mathcal{a}$, but possibly with different constants $b_j, c_j$ instead of $a_j$. This reduction will be achieved by the three lemmas below.

\begin{lemma} \label{LEMMA polynomial is const}
If the functions $\mathcal{a}, \mathcal{b}, \mathcal{c}$ contain an exponential term, the polynomial multiplying it must be constant.
\end{lemma}

\begin{proof}
See the appendix.
\end{proof}

\begin{lemma}
The functions $\mathcal{a}, \mathcal{b}, \mathcal{c}$ cannot contain two exponentials $e^{\lambda y}, e^{\mu y}$ with $\mu \neq \pm \lambda$.
\end{lemma}

\begin{proof}
Consider a typical exponential term in $\mathcal{a}, \mathcal{b}$ and $\mathcal{c}$ (due to Lemma~\ref{LEMMA polynomial is const} the polynomial multiplying it must be a constant), namely 

\begin{equation*}
\mathcal{a} \leftrightarrow a_0 e^{\lambda y}, \qquad
\mathcal{b} \leftrightarrow  b_0 e^{\lambda y}, \qquad
\mathcal{c} \leftrightarrow  c_0 e^{\lambda y} ,
\end{equation*}

\noindent where $a_0 \neq 0$. The equation coming from $e^{\lambda y}$ after substituting these forms into \eqref{R3} is (obtained analogously to the first equation of \eqref{singular k eqs y^2 and y} in the appendix)

\begin{equation*}
a_0 (e^{\lambda z} - 1) k'' + \left[2a_0\lambda + b_0 (e^{\lambda z} - 1) \right] k' + \left[ b_0 \lambda - a_0 \lambda^2 + c_0 (e^{\lambda z} - 1) \right] k = 0 .
\end{equation*}

\noindent After changing the variables $u(z) = k(z) (e^{\lambda z} - 1)$ it becomes 

\begin{equation} \label{singular k a0 eqn}
a_0 u'' + (b_0 -2a_0 \lambda) u' + (a_0 \lambda^2 - b_0 \lambda +c_0) u = 0 .
\end{equation}

\noindent Then, with $\nu = -\frac{b_0}{2a_0}$ and $\alpha_1, \alpha_2 \in \CC$ we have

\begin{equation} \label{singular k formula from exp term}
k(z) = \frac{e^{(\nu+\lambda) z} }{e^{\lambda z} - 1}\cdot
\begin{cases}
\alpha_1 z + \alpha_2, & \mu : = \sqrt{\tfrac{b_0^2}{4a_0^2} - \tfrac{c_0}{a_0}} = 0
\\
\alpha_1 \sinh(\mu z) + \alpha_2 \cosh(\mu z),  \qquad & \mu \neq 0
\end{cases}
\end{equation}

\noindent We claim that the set $\{\lambda, -\lambda\}$ is determined by the functions given above. In other words, up to the sign, $\lambda$ is determined by $k$. This will prove that in $\mathcal{a}(y)$, there cannot be another exponential $e^{\mu y}$ with $\mu \neq \pm \lambda$, because the equation coming from $e^{\mu y}$ will lead to a formula for $k$ incompatible with \eqref{singular k formula from exp term}. Computing the residue of $k$ at the pole $z=0$ we find $k_0 = \frac{\alpha_2}{\lambda}$, hence it is enough to show that $\alpha_2$ is determined up to the sign. Let $k$ be given by the second formula of \eqref{singular k formula from exp term} (in the other case the same argument will apply), write $\mu = \mu_1 + i\mu_2$ and $\lambda = \lambda_1 + i \lambda_2$. 

Let $\lambda_1 \neq 0$ and $\mu_1 \neq 0$, then w.l.o.g. we may assume $\mu_1 > 0$, otherwise negate $(\alpha_1 , \mu)$. If $\lambda_1 > 0$ we find

\begin{equation*}
k(z) \sim
\begin{cases}
\tfrac{1}{2} (\alpha_1 + \alpha_2) e^{(\nu + \mu)z}, & z \to +\infty,
\\
\tfrac{1}{2} (\alpha_1 - \alpha_2) e^{(\nu + \lambda - \mu)z}, & z \to -\infty.
\end{cases}
\end{equation*}

\noindent Therefore, $\alpha_2$ is equal to the difference of coefficients in the asymptotics of $k$ at plus and minus infinities. But when $\lambda_1 < 0$, by writing down the asymptotics, one can see that the same difference gives $-\alpha_2$.

Let now $\lambda_1 \neq 0$ and $\mu_1 = 0$, we find $k(z) \sim e^{\nu z} (i\alpha_1 \sin (\mu_2 z) + \alpha_2 \cos(\mu_2 z))$ as $z \to +\infty$ if $\lambda_1 > 0$, and when $\lambda_1 < 0$ the same formula holds, but the RHS multiplied by $- e^{\lambda z}$. Again we see that $\alpha_2$ is determined up to the sign.

Let $\lambda_1 = 0$ and $\mu_2 \neq 0$, we may assume $\mu_2 > 0$, otherwise negate $(\alpha_1, \mu)$, then 

\begin{equation*}
k(iz) \sim 
\begin{cases}
  \tfrac{1}{2} (\alpha_1 - \alpha_2) e^{i(\nu + \lambda - \mu)z}, & z \to +\infty,
\\
\tfrac{1}{2} (\alpha_1 + \alpha_2) e^{i(\nu + \mu)z}, & z \to -\infty.
\end{cases}
\end{equation*}

\noindent Finally, the case $\lambda_1 = \mu_2 = 0$ can be treated similarly.

Remains to note that $\mathcal{b}, \mathcal{c}$ cannot have an exponential $e^{\mu y}$ with $\mu \neq \pm \lambda$ either (we assume $a_0 e^{\lambda y}$ appears in $\mathcal{a}$). Indeed, if $\tilde{b}_0 e^{\mu y}$ and $\tilde{c}_0 e^{\mu y}$ appear in $\mathcal{b}$ and $\mathcal{c}$ respectively, then for $k$ we obtain an equation like \eqref{singular k a0 eqn}, but with $a_0 = 0$ and $b_0, c_0$ replaced with $\tilde{b}_0, \tilde{c}_0$, hence $k(z)=e^{(\mu + \tilde{\nu})z} / (e^{\mu z} - 1)$ with $\tilde{\nu} = - \tilde{c}_0 / \tilde{b}_0$. But this is of the same form as \eqref{singular k formula from exp term}, hence as we showed $\mu$ is determined up to its sign. In other words the two formulas for $k$ are compatible only if $\mu = \pm \lambda$.

\end{proof}

\begin{lemma}
The functions $\mathcal{a}, \mathcal{b}, \mathcal{c}$ cannot contain an exponential and a polynomial at the same time.
\end{lemma}

\begin{proof}
Let $a_5 e^{\lambda y} + \sum_{j=0}^4 a_j y^j$, with $a_5 \neq 0$ be part of $\mathcal{a}$. The functions $\mathcal{b}, \mathcal{c}$ also have such parts, but with possibly different constants $b_j, c_j$. From the above lemma we know that $k$ is given by \eqref{singular k formula from exp term} (with $a_0$ replaced by $a_5$). One can check that once these expressions for $\mathcal{a}, \mathcal{b}$ and $\mathcal{c}$ are substituted into \eqref{R3}, the factors $y^4$ get canceled and the equation corresponding to $y^3$ reads

\begin{equation} \label{esim}
a_4 z k'' + (b_4 z + 2a_4) k' + (c_4z +b_4)k = 0 .
\end{equation}  

\noindent Let us first show that $a_4 = 0$. For the sake of contradiction assume $a_4 \neq 0$, then the solution, with $\omega = -\frac{b_4}{2a_4}$, is given by

\begin{equation} \label{singular k pol eqn}
k(z) = \frac{e^{\omega z}}{z} \cdot
\begin{cases}
\beta_1z + \beta_2,  &\eta:=\sqrt{\tfrac{b_4^2}{4a_4^2} - \tfrac{c_4}{a_4}} = 0,
\\
\beta_1 \sinh(\eta z) + \beta_2 \cosh(\eta z), \qquad &\eta \neq 0.
\end{cases}
\end{equation}

\noindent We note that this is not compatible with \eqref{singular k formula from exp term}, because cross multiplying the two formulas we get (with $f,g$ being the second multiplying factors from \eqref{singular k formula from exp term} and \eqref{singular k pol eqn}, respectively)

\begin{equation*}
z e^{(\nu + \lambda) z} f(z) = e^{\omega z} (e^{\lambda z} -1) g(z) .
\end{equation*} 

\noindent If $g(z)=\beta_1 \sinh(\eta z) + \beta_2 \cosh(\eta z)$, we use the linear independence of $z e^{\gamma z}$ and $e^{\tilde{\gamma} z}$ to conclude that $k=0$. Let $g(z) = \beta_1z + \beta_2$, if $f$ is given by the first formula the above relation reads

\begin{equation*}
\alpha_1 z^2 e^{(\nu + \lambda)z} + \alpha_2 z e^{(\nu + \lambda)z} + \beta_1 z e^{\omega z} - \beta_1 z e^{(\omega + \lambda)z} = \beta_2 e^{(\omega + \lambda)z} - \beta_2 e^{\omega z} .
\end{equation*}

\noindent Because $\lambda \neq 0$, the exponentials on RHS are linearly independent, hence we conclude that $\beta_2 = 0$, which contradicts to $k$ having a pole at zero. When $f$ is given by the second formula the same argument applies. 

Thus, $a_4 = 0$, if $b_4 \neq 0$ we find $k(z) = e^{\omega z} / z$, but now $\omega = - c_4 / b_4$. This has the same form as \eqref{singular k pol eqn}, hence again it is incompatible with \eqref{singular k formula from exp term}. Therefore, $b_4 = 0$ and obviously $c_4 = 0$. With this information, the equation corresponding to $y^2$ is as \eqref{esim} with all subscripts changed from 4 to 3. Hence, the same procedure works and eventually we conclude $a_j = b_j = c_j = 0$ for $j=1,...,4$. 

\end{proof}

\subsection{Finding $k$}

The analysis of the previous subsection shows that we have two possible forms ($\lambda \neq 0 $)

\begin{equation*}
\text{I.} \ \mathcal{a}(y) = a_1 e^{\lambda y} + a_2 e^{-\lambda y} + a_0,
\qquad \qquad 
\text{II.} \ \mathcal{a}(y) =\sum_{j=0}^6 a_j y^j.
\end{equation*}

\noindent Moreover we also showed that in each case $\mathcal{b},\mathcal{c}$ are exactly of the same form as $\mathcal{a}$, only with possibly different constants $b_j, c_j$ instead of $a_j$.

\subsubsection{Case I}

\noindent Assume case I holds, substituting the expressions for $\mathcal{a}, \mathcal{b}, \mathcal{c}$ into \eqref{R3} we find that a linear combination of $e^{pm \lambda y}$ is zero, hence the coefficient of each exponential must vanish. Like this we obtain two ODEs for $k$. More precisely, 

\begin{equation*}
\begin{split}
a_1 (e^{\lambda z} - 1) k'' + \left[2a_1\lambda + b_1 (e^{\lambda z} - 1) \right] k' + \left[ b_1 \lambda - a_1 \lambda^2 + c_1 (e^{\lambda z} - 1) \right] k = 0,
\\
a_2 (e^{-\lambda z} - 1) k'' + \left[-2a_2\lambda + b_2 (e^{-\lambda z} - 1) \right] k' + \left[ -b_2 \lambda - a_2 \lambda^2 + c_2 (e^{-\lambda z} - 1) \right] k = 0 .
\end{split}
\end{equation*}

\noindent Note that the second equation is obtained from the first one if we negate $\lambda$ and change the subscripts of $a_1,b_1,c_1$ from 1 to 2. Consider the following cases:

\vspace{.1in}

\noindent \textbf{Case I.1.} $a_1=a_2=0$, then $\mathcal{a}\equiv 0$ and from the boundary conditions $\mathcal{b}(\pm 1) = 0$. W.l.o.g. let $b_1 \neq 0$ solving the first ODE for $k$ we get, with $\nu = -\frac{c_1}{b_1}$

\begin{equation*}
k(z) = \frac{e^{(\nu + \lambda)z}}{e^{\lambda z} -1} = \frac{e^{\left(\nu + \frac{\lambda}{2}\right)z}}{2 \sinh \left( \frac{\lambda}{2}z \right)} .
\end{equation*}

\noindent For this to satisfy also the second ODE we need $c_2 = -(\nu + \lambda) b_2$. One can check that for $k$ to be smooth in $[-2,2] \backslash \{0\}$, we cannot have $\lambda = \pi i n$, therefore the boundary conditions on $\mathcal{b}$ imply $b_1 = b_2$ and so $\mathcal{b}(y) = \cosh (\lambda y) - \cosh \lambda$. Now if $\lambda \in i\RR$, for the same reason we require $|\lambda|<\pi$. From the relation \eqref{c in terms of b singular comm*} we see that $\mathcal{c}(y) = \frac{1}{2} \mathcal{b}'(y)$. After ignoring the exponential in the numerator of the formula for $k$ (see Remark~\ref{REM multiplier}) we obtain

\begin{equation} \label{item 2 alpha 0}
k(z) =\frac{1}{\sinh \left( \frac{\lambda}{2}z \right)}, \qquad \qquad
\begin{cases}
\mathcal{a}(y) = 0,
\\
\mathcal{b}(y) = \cosh (\lambda y) - \cosh \lambda,
\\
\mathcal{c}(y) = \tfrac{1}{2} \mathcal{b}'(y).
\end{cases}
\end{equation}

\vspace{.1in}

\noindent \textbf{Case I.2.} If $a_1 \neq 0$ (the case $a_2 \neq 0$ can be treated analogously) by rescaling let us take $a_1 = \frac{1}{2}$, then as the formula \eqref{singular k formula from exp term} was obtained we get, by w.l.o.g. choosing $\nu = - \lambda / 2$, or equivalently $b_1 = \lambda a_1$ (see Remark~\ref{REM multiplier}) that

\begin{equation*}
k(z) = \frac{1}{\sinh \left( \frac{\lambda}{2}z \right)}\cdot
\begin{cases}
\alpha_1 z + \alpha_2, & \mu : = \sqrt{b_1^2 - 2c_1} = 0,
\\
\alpha_1 \sinh(\mu z) + \alpha_2 \cosh(\mu z),  \qquad & \mu \neq 0.
\end{cases}
\end{equation*}

$\bullet$ Let $k$ be given by the first formula. It is easy to check that $\lambda = \pi i n$, with $n \in \ZZ$ contradicts to the smoothness assumption on $k$, so the  boundary conditions imply that $a_1= a_2$ and therefore $\mathcal{a}(y) = \cosh(\lambda y) - \cosh \lambda$. Because of the same reason, when $\lambda \in i\RR$ we need a further restriction $|\lambda| < \pi$. The boundary conditions $\mathcal{b}(\pm 1) = \mathcal{a}'(\pm 1)$ then imply

\begin{equation*}
b_2 = -\tfrac{\lambda}{2}, \quad b_0 = 0 \quad \Rightarrow \quad \mathcal{b}(y) = \tfrac{\lambda}{2} e^{\lambda y} - \tfrac{\lambda}{2} e^{-\lambda y} = \mathcal{a}'(y) .
\end{equation*}

\noindent Now, $k$ has to satisfy also the second ODE, so we substitute the expression for $k$ there and simplify the result to find

\begin{equation*}
e^{-\frac{\lambda}{2}z} (\alpha_1 z + \alpha_2) \left( c_2 - \tfrac{\lambda^2}{8} \right) = 0 ,
\end{equation*}

\noindent which clearly implies $c_2 = \tfrac{\lambda^2}{8}$. But because this was the case $\mu = 0$ we have $c_1 = \frac{b_1^2}{2} = \tfrac{\lambda^2}{8}$ and therefore we conclude that $\mathcal{c}(y) = \tfrac{\lambda^2}{2} \mathcal{a}(y)$. Thus, we proved \eqref{k general} and \eqref{a,b,c general} of Theorem~\ref{THM commutation} in the limiting case $\mu = 0$. Moreover, when $\alpha_1 = 0$ we obtain the same kernel as in \eqref{item 2 alpha 0}, hence we can take a linear combination of the differential operator of this case and the one in \eqref{item 2 alpha 0} and $K$ will still commute with it. This proves item 2 of Theorem~\ref{THM commutation}. 

\vspace{.1in}  

$\bullet$ Let $k$ be given by the second formula. When $\lambda \in i\RR$ there are further restrictions for parameters. Let us analyze them. Firstly, if $\lambda \in i\RR$ with $|\lambda| \geq 2\pi$, then the denominator of $k$ has zeros at $\pm \frac{2\pi i}{\lambda}, \pm \frac{4\pi i}{\lambda} \in [-2,2]$, which cannot be canceled out by the numerator, therefore $|\lambda|<2\pi$. So there are two cases: when $|\lambda|<\pi$, \ $k$ is smooth in $[-2,2] \backslash \{0\}$ and when $\pi \leq |\lambda|<2\pi$ the denominator of $k$ has zeros at $\pm \frac{2\pi i}{\lambda} \in [-2,2]$, which can be canceled out by the numerator \IFF $\alpha_1 = 0$ and $\cosh \left( \frac{2 \pi i \mu}{\lambda} \right) = 0$, i.e. $\mu = \lambda \frac{2m+1}{4}$ for some $m \in \ZZ$. This is summarized in Remark~\ref{REM lambda in iR}.

Let us substitute the expression for $k$ into the second ODE, multiply the result by $e^{\frac{\lambda}{2}z}$. After simplification we obtain

\begin{equation*}
\begin{split}
\left[ (\mu^2 a_2 + \tfrac{\lambda^2 a_2}{4} + \tfrac{b_2 \lambda}{2} +c_2) \alpha_1 + \mu \alpha_2 (a_2 \lambda + b_2) \right] \sinh(\mu z) +&
\\
+ \left[ (\mu^2 a_2 + \tfrac{\lambda^2 a_2}{4} + \tfrac{b_2 \lambda}{2} +c_2) \alpha_2 + \mu \alpha_1 (a_2 \lambda + b_2) \right] \cosh(\mu z)& = 0 .
\end{split}
\end{equation*}

\noindent By linear independence we conclude that the coefficients of $\sinh(\mu z), \cosh(\mu z)$ must be zero. Or equivalently their sum and difference must be zero, but these equations can be written as

\begin{equation} \label{coefficient system from 2nd ODE}
\begin{cases}
(\alpha_1 + \alpha_2) \left( (\mu + \frac{\lambda}{2}) [(\mu + \frac{\lambda}{2})a_2 + b_2] + c_2 \right) = 0,
\\
(\alpha_1 - \alpha_2) \left( (\mu - \frac{\lambda}{2}) [(\mu - \frac{\lambda}{2})a_2 - b_2] + c_2 \right) = 0 .
\end{cases}
\end{equation}

\noindent The boundary conditions $\mathcal{a}(\pm 1) = 0$ imply that $a_0 = -a_1 e^\lambda - a_2 e^{-\lambda}$ and

$$(a_1 - a_2) (e^\lambda - e^{-\lambda}) = 0 .$$

$a)$ Let $a_2 = a_1$, then $\mathcal{a}(y) = \cosh(\lambda y) - \cosh(\lambda)$ and from the boundary conditions $\mathcal{b}(\pm 1) = \mathcal{a}'(\pm 1)$ we find $\mathcal{b}(y) = \mathcal{a}'(y)$ as was discussed above. Now in this case \eqref{coefficient system from 2nd ODE} simplifies to

\begin{equation*}
\begin{cases}
(\alpha_1 + \alpha_2) \left( \tfrac{\lambda^2}{4} - \mu^2 - 2c_2 \right) = 0,
\\
(\alpha_1 - \alpha_2) \left( \tfrac{\lambda^2}{4} - \mu^2 - 2c_2 \right) = 0 .
\end{cases}
\end{equation*}

\noindent But because both $\alpha_1, \alpha_2$ are not zero at the same time, we get $c_2 = \tfrac{1}{2} (\tfrac{\lambda^2}{4} - \mu^2)$. From the definition of $\mu$ we see that also $c_1 = \tfrac{1}{2} (\tfrac{\lambda^2}{4} - \mu^2)$. And using the freedom of choosing $c_0$ we conclude that we may write $\mathcal{c}(y) = (\tfrac{\lambda^2}{4} - \mu^2) \mathcal{a}(y)$. This proves \eqref{k general} and \eqref{a,b,c general} of Theorem~\ref{THM commutation} in the case $\mu \neq 0$.

$b)$ Let $e^\lambda = e^{-\lambda}$, i.e. $\lambda = \pi i n$ for some $n \in \ZZ$. But the above discussion implies that $\alpha_1 = 0$, $\lambda = \pi i$ (or $-\pi i$, but this would lead to the same results) and $\mu = \lambda \frac{2m+1}{4}$ with $m \in \ZZ$. In this case \eqref{coefficient system from 2nd ODE} implies

\begin{equation*}
b_2 = - \lambda a_2, \qquad \qquad c_2=a_2 \left( \tfrac{\lambda^2}{4} - \mu^2 \right) .
\end{equation*}

\noindent Recalling that $b_1 = \lambda a_1$, the boundary conditions $\mathcal{b}(\pm 1) = \mathcal{a}'(\pm 1)$ imply $b_0 = 0$ and so far we have $\mathcal{a}(y) = a_1 (e^{\lambda y} - e^\lambda)  + a_2 (e^{-\lambda y} - e^{-\lambda})$ and $\mathcal{b}(y) = \mathcal{a}'(y)$. Finally, again from the definition of $\mu$ we have $c_1 = a_1 (\tfrac{\lambda^2}{4} - \mu^2)$. This and the above formula for $c_2$ (and the freedom of choosing $c_0$) allow one to write $\mathcal{c}(y) = (\tfrac{\lambda^2}{4} - \mu^2) \mathcal{a}(y)$. This proves item 1 of Theorem~\ref{THM commutation}. Of course to start with we assumed $a_1 \neq 0$ and we normalized $a_1 = \tfrac{1}{2}$, but when considering the case $a_2 \neq 0$ we can allow $a_1$ to vanish. This explains why there are no restrictions on $\alpha, \beta$ in item 1 of Theorem~\ref{THM commutation}.

\subsubsection{Case II}

\noindent Assume case II holds, substituting the expressions for $\mathcal{a}, \mathcal{b}, \mathcal{c}$ into \eqref{R3} we find that a linear combination of monomials $y^j$ is zero, hence the coefficient of each $y^j$ must vanish (one can check that $y^6$ cancels out). These relations can be conveniently written as

\begin{equation} \label{singular comm relations for k}
\begin{split}
\left[ \frac{\mathcal{a}^{(j)}(z)}{j!} - a_j \right] k'' +\left[ \frac{\mathcal{b}^{(j)}(z)}{j!} - b_j + 2(j+1)a_{j+1} \right] k' + &
\\
+\left[ \frac{\mathcal{c}^{(j)}(z)}{j!} - c_j + (j+1)b_{j+1} - (j+1)(j+2)a_{j+2} \right]k &= 0, \qquad j=0,...,5 ,
\end{split}
\end{equation}

\noindent with the convention that $a_7 = 0$. Let $\deg (\mathcal{a}) = m,\ \deg (\mathcal{b}) = n $ and $\deg(\mathcal{c})=s$.

\vspace{.1in}

\noindent \textbf{Case II.1.} Let $\mathcal{a} \equiv 0$, then $\mathcal{b}(\pm 1) = 0$ and hence $n \geq 2$. By scaling we let $b_n = 1$. We are going to show that $n$ cannot be strictly larger than 2 and so $n=2$. Note that $s \leq n$, otherwise the above relation with $j=s - 1$ reads $c_s z k = 0$, which implies $k=0$ since $c_s \neq 0$ by the definition of $s$. Now  \eqref{singular comm relations for k} with $j=n - 1$ reads

\begin{equation} \label{singular comm a=0 ode for k}
z k' + [1 + c_n z ] k = 0 ,
\end{equation}

\noindent whose solution is given by $k(z)=\alpha \frac{e^{-c_n z}}{z }$, where $\alpha \in \CC$. Invoking Remark~\ref{REM multiplier} we may w.l.o.g. assume $c_n = 0$. The relation with $j=n-2$ becomes

\begin{equation*}
\left[ \tfrac{n}{2} z^2 + b_{n-1} z \right] k'+\left[c_{n-1} z+b_{n-1} \right] k=0 .
\end{equation*}

\noindent Substituting $k(z) = \frac{1}{z}$ into this equation we obtain  $c_{n-1} = \frac{n}{2}$. Now, if $n>2$ we consider the relation for $j=n-3$, which reads

 $$\left[ \tfrac{n(n-1)}{6} z^3 + \tfrac{n-1}{2} b_{n-1} z^2 + b_{n-2} z \right] k' + \left[ \tfrac{n-1}{2} c_{n-1} z^2 + c_{n-2} z + b_{n-2} \right] k =0 .$$ 
 
\noindent Again substituting the expression for $k$ and using the expression for $c_{n-1}$ we obtain

\begin{equation} \label{contradiction}
\tfrac{n(n-1)}{12} z + c_{n-2} + \tfrac{n-1}{2} b_{n-1} = 0 ,
\end{equation}

\noindent which is a contradiction. Thus our conclusion is that $n=2$, in which case $\mathcal{b}(y) = y^2 - 1$, $c_2 = 0$, $c_1 = 1$ and hence $\mathcal{c}(y) = y$, and we obtain the operator in item 4 of Theorem~\ref{THM commutation} when $\mathcal{p}=0$.

\vspace{.1in}

\noindent \textbf{Case II.2.} Let $\mathcal{a} \neq 0$, then $m \geq 2$. By scaling we let $a_m = 1$. Let us first show that $n \leq m$. For the sake of contradiction assume $n>m$. If also $s>n$, then \eqref{singular comm relations for k} with $j=s-1$ reads $c_s z k =0$, which is a contradiction and therefore $s \leq n$. Now \eqref{singular comm relations for k} with $j=n-1$ reads 

\begin{equation*}
z k' + \left[ 1 + c_n z \right] k = 0 ,
\end{equation*}

\noindent with the convention that $c_n = 0$ if $s<n$. As in the previous case w.l.o.g. we assume $c_n = 0$ so that $k(z) = \frac{1}{z}$. Using these and looking at \eqref{singular comm relations for k} for $j=n-2$ and $j=n-3$ we obtain exactly the same contradiction \eqref{contradiction} as in the previous case (only with a different free constant). 

Thus $n \leq m$, and it is easy to see that also $s \leq m$. The relation \eqref{singular comm relations for k} for $j=m-1$ reads

\begin{equation} \label{sing k m-1}
zk''+ (2+b_m z) k' + (b_m + c_m z) k = 0 ,
\end{equation}

\noindent whose solution is, with $\alpha_1, \alpha_2 \in \CC$

\begin{equation} \label{sing k m-1 sol}
k(z) = \frac{e^{-\frac{b_m}{2}z}}{z} \cdot
\begin{cases}
\alpha_1 \sinh(\mu z) + \alpha_2 \cosh(\mu z), \qquad &\mu^2:=\tfrac{b_m^2}{4} - c_m \neq 0,
\\
\alpha_1 z + \alpha_2, \qquad &\mu=0.
\end{cases}
\end{equation}

\noindent Invoking Remark~\ref{REM multiplier} let us w.l.o.g. assume $b_m = 0$. Then from \eqref{sing k m-1}

\begin{equation} \label{k'' in terms of k' and k}
k''(z) = -\frac{2 k'(z) + c_m z k(z)}{z} .
\end{equation}

\noindent  The relation \eqref{singular comm relations for k} for $j=m-2$ (after dividing it by $m-1$) is

\begin{equation*}
\begin{split}
\left[ a_{m-1} z + \tfrac{m}{2}z^2 \right] k'' +  \left(b_{m-1} z +2 a_{m-1} \right) k'
+  \left[ c_{m-1} z + \tfrac{m}{2} c_m z^2 + b_{m-1} - m \right] k = 0 .
\end{split}
\end{equation*}

\noindent Substituting $k''$ from \eqref{k'' in terms of k' and k} into this equation we obtain

\begin{equation}\label{sing k m-2}
(b_{m-1}-m) z k' + \left[ (c_{m-1} - c_m a_{m-1}) z + b_{m-1}-m \right]k = 0 .
\end{equation}

\noindent Let us now consider the cases for different values of $m$:

\vspace{.1in}
\begin{itemize}
\item[a)]
 let $m=2$, then $\mathcal{a}(y)=y^2-1$, $b_2 = 0$ further the boundary conditions imply $b_1=2$, $b_0 = 0$ and hence $\mathcal{b}(y)=2y$. Then \eqref{sing k m-2} reads $c_1 k=0$, hence $c_1 = 0$ and so $\mathcal{c}(y) = c_2 y^2$. $k(z)$ is determined from \eqref{sing k m-1 sol}, where $\mu^2 = -c_2$. This proves formulas \eqref{k general} and \eqref{a,b,c general} of Theorem~\ref{THM commutation} in the limiting case $\lambda = 0$.

\vspace{.1in}

\item[b)] let $m=3$, then $\mathcal{a}(y) = (y^2-1)(y-\sigma)$ and $b_3 = 0$. In particular we see that $a_2 = -\sigma$ and $a_1 = -1$. From the boundary conditions $b_0 = 2-b_2; \ b_1 = - 2 \sigma = 2a_2$. The relation \eqref{singular comm relations for k} with $j=m-3 = 0$ reads

\begin{equation*}
(z^3 + a_2 z^2 + a_1 z) k'' + (b_2 z^2 + b_1 z + 2a_1) k' + (c_3z^3 + c_2 z^2 + c_1 z) k = 0 .
\end{equation*}

\noindent Substituting $k''$ from \eqref{k'' in terms of k' and k} this simplifies to

\begin{equation*}
(b_2-2) z^2 k' + [(c_2-c_3 a_2)z^2 + (c_1+c_3) z] k = 0 ,
\end{equation*}

\noindent and combining this with \eqref{sing k m-2} we obtain

\begin{equation*}
z k' + \left( c_1+c_3-b_2+3 \right)k = 0 .
\end{equation*}

\noindent But because $k$ has a simple pole at $0$, we must have $c_1+c_3-b_2+3=1$, hence $c_3 = b_2-c_1-2$. Then $k(z) = 1 / z$, substituting this expression into \eqref{sing k m-1} we conclude $c_1 = b_2-2$ and hence $c_3 = 0$. Next we substitute it into \eqref{sing k m-2} to find $c_2 = 0$. Thus 

\begin{equation*}
\begin{cases}
\mathcal{a}(y)=(y^2-1)(y-\sigma) 
\\
\mathcal{b}(y)= b_2 y^2-2\sigma y + 2 - b_2
\\
\mathcal{c}(y)=(b_2 - 2)y 
\end{cases}
\end{equation*}

\noindent This proves item 4 of Theorem~\ref{THM commutation}, when $\beta = b_2 - 3$ and $\mathcal{p}$ is a first order polynomial.

\vspace{.1in}

\item[c)] let $m=4$, then $\mathcal{a}(y) = (y^2-1)(y-\sigma_1)(y-\sigma_2)$, \ $b_4 = 0$. Note that $a_3 = - \sigma_1 - \sigma_2; a_2 = \sigma_1 \sigma_2 - 1$. Further, from the boundary conditions on $\mathcal{b}$ we get $b_1 = 2(a_2+2) -b_3$ and $b_0 = -b_2+2a_3$.  From \eqref{sing k m-1 sol} $k$ has two possible forms, assume first $k(z) = \frac{1}{z} (\alpha_1 z + \alpha_2)$ in which case $c_4 = \frac{b_4^2}{4} = 0$. Since $k$ has a simple pole at the origin $\alpha_2 \neq 0$ and let us normalize $\alpha_2 = 1$. \eqref{sing k m-2} in this case reads $(b_3-4) z k' + (c_3z +b_3-4)k = 0$. Substituting the expression for $k$ into this equation we obtain

$$c_3\alpha_1 z + c_3 + (b_3-4) \alpha_1 = 0 ,$$

\noindent which implies that $c_3 = 0$ and

\begin{equation} \label{2 cases}
\alpha_1 (b_3-4)=0 .
\end{equation}

\noindent The relations \eqref{singular comm relations for k} with $j=m-3$ and $j=m-4$ read respectively as

\begin{equation} \label{sing k m-3}
(4z^3 + 3a_3z^2 + 2a_2 z) k'' + (3b_3z^2 + 2b_2 z + 4a_2) k' + 2 (c_2z - 3a_3 + b_2)k = 0 ,
\end{equation}

\begin{equation} \label{sing k m-4}
(z^4 + a_3z^3 + a_2 z^2 + a_1 z) k'' + (b_3z^3 + b_2 z^2 + b_1 z+2a_1) k' + (c_2z^2 + c_1z - 2a_2 + b_1)k = 0 .
\end{equation}

\noindent Now, \eqref{2 cases} implies that we should consider two cases:

$\bullet$ If $\alpha_1 = 0$, we substitute $k(z) = \frac{1}{z}$ into \eqref{sing k m-3} and find $c_2 =\frac{3}{2} b_3 -4$. Finally substitution into \eqref{sing k m-4} gives 

$$\frac{b_3-4}{2} z + 2a_3 - b_2 + c_1 = 0 ,$$

\noindent therefore $b_3 = 4$ and $c_1 = -2a_3 +b_2$. Putting everything together we obtain

\begin{equation*}
\begin{cases}
\mathcal{a}(y)=(y^2-1)(y-\sigma_1)(y-\sigma_2) 
\\
\mathcal{b}(y)= 4y^3 + b_2y^2 + 2(\sigma_1 \sigma_2 - 1)y -b_2-2(\sigma_1  + \sigma_2)
\\
\mathcal{c}(y)=2y^2 + (b_2 +2 \sigma_1 + 2 \sigma_2)y
\end{cases}
\end{equation*}

\noindent This proves item 4 of Theorem~\ref{THM commutation}, when $\beta = b_2 + 3(\sigma_1 + \sigma_2)$ and $\mathcal{p}$ is a second order polynomial.

$\bullet$ If $\alpha_1 \neq 0$, we get $b_3 = 4$, substituting $k(z) = \alpha_1 + \frac{1}{z}$ into \eqref{sing k m-3} we obtain

$$c_2 \alpha_1 z + (b_2-3a_3)\alpha_1 + c_2 - 2 = 0 ,$$

\noindent hence we deduce $c_2 = 0$ and $\alpha_1 (b_2 -3a_3) = 2$. Finally, we substitute $k$ into \eqref{sing k m-4} and obtain $c_1 = -3a_3  + b_2$ and $a_3 (b_2 -3a_3) = 0$, but because $b_2-3a_3 \neq 0$ we get $a_3=0$, i.e. $\sigma_1 = - \sigma_2$. Then also $\alpha_1 = \frac{2}{b_2}$, \ $\displaystyle k(z) = \frac{2}{b_2} + \frac{1}{z}$ and

\begin{equation*}
\begin{cases}
\mathcal{a}(y)=(y^2-1)(y^2 - \sigma_1^2) ,
\\
\mathcal{b}(y)= 4y^3 + b_2y^2 - 2(\sigma_1^2 + 1)y -b_2,
\\
\mathcal{c}(y)= b_2 y .
\end{cases}
\end{equation*}

\noindent This establishes item 3 of Theorem~\ref{THM commutation} with $\beta = b_2/2$.

Let now $k(z) = \frac{1}{z} (\alpha_1 \sinh(\mu z) + \alpha_2 \cosh(\mu z))$, with $\mu^2 = - c_4 \neq 0$. One can check by subsequent substitutions into \eqref{sing k m-2}, \eqref{sing k m-3} and \eqref{sing k m-4} that this case is impossible.

\vspace{.1in}

\item[d)] Subsequent substitutions show also that $m \geq 5$ is impossible.
\end{itemize}

\medskip

\noindent\textbf{Acknowledgments.}
This material is based upon work supported by the National Science Foundation under Grant No. DMS-1714287.

\section{Appendix}

Here we prove Lemma~\ref{LEMMA polynomial is const}, stating that if the functions $\mathcal{a}, \mathcal{b}, \mathcal{c}$ contain an exponential term, the polynomial multiplying it must be a constant. So let us concentrate on a typical exponential term in $\mathcal{a}, \mathcal{b}$ and $\mathcal{c}$, namely 

\begin{equation*}
\mathcal{a} \leftrightarrow e^{\lambda y} \sum_{j=0}^2 a_j y^j, \qquad
\mathcal{b} \leftrightarrow  e^{\lambda y} \sum_{j=0}^3 b_j y^j, \qquad
\mathcal{c} \leftrightarrow  e^{\lambda y} \sum_{j=0}^3 c_j y^j .
\end{equation*}

\noindent The goal is to show that all the coefficients vanish, except possibly for $a_0, b_0, c_0$. We are going to substitute these expressions into \eqref{R3}. The result becomes a linear combination of terms $y^j e^{\lambda y}$, hence the coefficient of each such terms must vanish. Below we analyze these coefficients, which are in fact ODEs for $k$. 

\vspace{.1in}  

$1.$ First let us show that the polynomials in $\mathcal{b}$ and $\mathcal{c}$ cannot be of higher order, than the polynomial in $\mathcal{a}$, i.e. $b_3 = c_3 = 0$. The equations corresponding to $y^3 e^{\lambda y}$ and $y^2 e^{\lambda y}$ are

\begin{equation} \label{singular k eqs y^3 and y^2}
\begin{split}
 b_3 (e^{\lambda z} - 1)  k' + \left[ b_3 \lambda + c_3 (e^{\lambda z} - 1) \right] k &= 0 ,
\\[1ex]
3(b_3k'+c_3 k) e^{\lambda z} z + (a_2k''+b_2k'+c_2k)e^{\lambda z}
+(2\lambda a_2 -b_2)k'
- a_2k'' -
\\
-[\lambda^2 a_2  -b_2 \lambda +c_2 -3b_3] k &=0 .
\end{split}
\end{equation}

\noindent Assume $b_3 \neq 0$, from the first equation $k(z) = e^{\left( \lambda - \frac{c_3}{b_3} \right) z} / (e^{\lambda z} -1) $. Invoking Remark~\ref{REM multiplier} w.l.o.g. we assume $c_3 = \lambda b_3$ in which case $k(z) = 1/(e^{\lambda z} -1)$. Substitute this into the second equation and multiplying the result by $(e^{\lambda z} -1)^2$ we obtain

\begin{equation*}
(a_2 \lambda^2 - b_2 \lambda +c_2) e^{2\lambda z} + (2b_2 \lambda - 2a_2 \lambda^2 + 3b_3 - 2c_2) e^{\lambda z} - 3b_3 \lambda z e^{\lambda z} + a_2 \lambda^2 - b_2 \lambda +c_2 - 3b_3 = 0 .
\end{equation*}

\noindent The functions $e^{2\lambda z}, e^{\lambda z}, z e^{\lambda z}$ and $1$ are linearly independent, hence the coefficient of each one must vanish. But we see that the coefficient of $z e^{\lambda z}$ is $3b_3 \lambda \neq 0$, which is a contradiction. Thus, $b_3 = 0$ and therefore also $c_3 = 0$.

\vspace{.1in}

2. We now show that $a_2 = 0$. The equations corresponding to $y^2 e^{\lambda y}$ and $y e^{\lambda y}$ are

\begin{equation} \label{singular k eqs y^2 and y}
\begin{split}
a_2 (e^{\lambda z} - 1) k'' + \left[2a_2\lambda + b_2 (e^{\lambda z} - 1) \right] k' + \left[ b_2 \lambda - a_2 \lambda^2 + c_2 (e^{\lambda z} - 1) \right] k &= 0 ,
\\[1ex]
2(a_2 k'' + b_2k'+c_2 k) e^{\lambda z} z + (a_1k''+b_1k'+c_1k)e^{\lambda z}
+(2\lambda a_1 +4a_2-b_1)k' -
\\
- a_1k'' -[\lambda^2 a_1 + (4a_2 -b_1) \lambda +c_1 -2b_2] k &=0 .
\end{split}
\end{equation}

\noindent Assume $a_2 \neq 0$, and by normalization let us assume $a_2=1$.  Solving the first equation we get (as was done in \eqref{singular k formula from exp term})

\begin{equation*}
k(z) = \frac{e^{\left( \lambda -\frac{b_2}{2} \right) z} }{e^{\lambda z} - 1}\cdot
\begin{cases}
\alpha_1 z + \alpha_2, & \mu : = \sqrt{\tfrac{b_2^2}{4} - c_2} = 0
\\
\alpha_1 e^{\mu z} + \alpha_2 e^{-\mu z},  \qquad & \mu \neq 0
\end{cases}
\end{equation*}

\noindent Using Remark~\ref{REM multiplier} let us w.l.o.g. assume $b_2 = 2 \lambda$.

 Let $k$ be given by the first formula. Since $\alpha_2 \neq 0$ we may normalize it to be one, so $k(z) = \frac{\alpha_1 z + \alpha_2 }{e^{\lambda z} - 1}$ and $c_2 = \frac{b_2^2}{4}$. Substituting this expression into the second equation of \eqref{singular k eqs y^2 and y} and multiplying the result by $(e^{\lambda z} - 1)^3$ we obtain

\begin{equation*}
\begin{split}
(p_1z + p_2) e^{3 \lambda z} + \left[ 2 \lambda^2 \alpha_1 z^2 + \left( (2 - 3\alpha_1 a_1) \lambda^2 + (3b_1-8) \alpha_1 \lambda - 3c_1 \alpha_1 \right)z + p_3 \right] e^{2\lambda z} +&
\\
+(p_4 z^2 + p_5 z + p_6) e^{\lambda z} + p_7 z + p_8 &= 0 ,
\end{split}
\end{equation*}

\noindent where $p_j$ are constants depending on $a_1, b_1, c_1, \alpha_1, \lambda$ and their particular expressions are not important. From linear independence the coefficient of $z^2 e^{2\lambda z}$ must vanish, which implies $\alpha_1 = 0$, but then the coefficient of $z e^{2\lambda z}$ becomes $2 \lambda^2 \neq 0$, which leads to a contradiction.

Let $k$ be given by the second formula, then $c_2 = \frac{b_2^2}{4} - \mu^2$ and $\mu \neq 0$. Substituting $k$ into the second equation of \eqref{singular k eqs y^2 and y} and multiplying the result by $e^{\mu z} (e^{\lambda z} - 1)^3$ we obtain

\begin{equation} \label{last}
\begin{split}
\alpha_1 (\mu + \tfrac{\lambda}{2}) z e^{(2\mu + \lambda)z} - \alpha_1 (\mu- \tfrac{\lambda}{2}) z e^{(2\mu + 2\lambda)z} + \alpha_2 (\mu + \tfrac{\lambda}{2}) z e^{2\lambda z} - \alpha_2  (\mu- \tfrac{\lambda}{2}) z  e^{\lambda z}  = 
\\
= q_0 + q_1 e^{\lambda z} + q_2 e^{2\lambda z} + q_3 e^{3\lambda z} + q_4 e^{2\mu z} + q_5 e^{(2\mu + \lambda)z} + q_6 e^{(2\mu + 2\lambda)z} + q_7 e^{(2\mu + 3\lambda)z} ,
\end{split}
\end{equation}

\noindent where $q_j$ are constants whose particular expressions are not important. Note that the functions on LHS of \eqref{last} are linearly independent from the ones on RHS. If all the exponents on LHS are distinct then the coefficients multiplying them must be zero. In particular $\alpha_1 (\mu + \tfrac{\lambda}{2}) = 0$ and $\alpha_1 (\mu- \tfrac{\lambda}{2}) = 0$, which imply $\alpha_1 = 0$. Analogously, $\alpha_2 = 0$ leading to $k=0$. Now assume the exponents on LHS of \eqref{last} are not distinct, then there are two possibilities:

\begin{enumerate}
\item[a)] $2\mu + \lambda = 2\lambda$, hence $\lambda = 2\mu$ and LHS of \eqref{last} becomes $2\mu (\alpha_1 + \alpha_2) z e^{4\mu z}$. Hence $\alpha_1 = -\alpha_2$, which then implies

$$k(z) = \frac{2\alpha_1 \sinh(\mu z)}{e^{\lambda z} -1} .$$

\noindent This contradicts to the assumption that $k$ has a simple pole at the origin.

\item[b)] $2\mu + 2\lambda = \lambda$, hence $\lambda = -2\mu$. Similarly, this case also leads to a contradiction.  

\end{enumerate}

\vspace{.1in}

3. To show $b_2 = c_2 = 0$, we can apply the same argument of 1, because once we established $a_2 = 0$ the equations in \eqref{singular k eqs y^2 and y} are exactly the ones in \eqref{singular k eqs y^3 and y^2}, the only difference is that in the latter we need to replace $b_3, c_3$ by $\frac{2}{3} b_2, \frac{2}{3} c_2$ and $a_2, b_2 ,c_2$ by $a_1, b_1, c_1$ respectively. After this, in an analogous way to 2, we show that $a_1 = 0$, again the equations corresponding to $y e^{\lambda y}$ and $e^{\lambda y}$ are exactly the ones in \eqref{singular k eqs y^2 and y} only $a_2, b_2, c_2$ need to be replaced by $\frac{a_1}{2}, \frac{b_1}{2}, \frac{c_1}{2}$ and $a_1, b_1, c_1$ by $a_0, b_0,c_0$ respectively. Finally, again as in 1, we establish that also $b_1 = c_1 = 0$.

\end{document}